\documentclass{amsart}

\newtheorem{thm}{Theorem}[section]
\newtheorem{cor}[thm]{Corollary}

\theoremstyle{definition}

\theoremstyle{remark}
\newtheorem{rem}[thm]{Remark}

\begin{document}
\title[Inequalities for generalized normalized $\delta$-Casorati curvatures]{Inequalities for generalized normalized $\delta$-Casorati curvatures of slant submanifolds in quaternionic space forms}

\author[Jaewon Lee, Gabriel-Eduard V\^{\i}lcu]{Jaewon Lee, Gabriel-Eduard V\^{\i}lcu}

\date{}
\maketitle

\abstract  In this paper we prove two sharp inequalities involving the normalized scalar curvature and the generalized normalized $\delta$-Casorati curvatures for slant submanifolds in quaternionic space forms. We also characterize those submanifolds for which the equality cases hold. These results are a
generalization of some recent results concerning the Casorati curvature for a slant submanifold in a quaternionic space
form obtained by Slesar et al.: \emph{J. Inequal. Appl.} 2014, 2014:123.
\endabstract

\section{Introduction}

In order to provide answers to an open question raised by S.S. Chern \cite{CSS} concerning the existence of minimal immersions into
Euclidean spaces of arbitrary dimension, Prof. B.-Y. Chen \cite{CH1} introduced in the early 1990's
new types of Riemannian invariants, known in the literature as Chen invariants or $\delta$-invariants and established general optimal inequalities involving the new intrinsic invariants and the main extrinsic invariants for arbitrary Riemannian
submanifolds. Thus was born the theory of Chen invariants, one of the most interesting research topic in differential geometry of submanifolds.

After $\delta$-invariants were invented and first inequalities were proved, such invariants and Chen-like inequalities were considered in different ambient spaces for many classes of submanifolds. For example, new optimal inequalities involving Chen invariants were recently proved in \cite{ABM,SCD,CH12,CDVV,CPW,DNG,GKK1,GKK2,OM,MHS}. We also note that some interesting inequalities for the length of the second fundamental form of the warped product submanifolds were obtained recently in \cite{AK,AK2,AK3,AK4,KUS,UK}.

On the other hand, it is well-known that the Casorati curvature of a submanifold in a Riemannian manifold is an extrinsic invariant defined as the normalized square of the length of the second fundamental form and it was preferred by Casorati over the traditional Gauss curvature because corresponds better with the common intuition of curvature \cite{CAS}. Some optimal inequalities involving Casorati curvatures were proved in \cite{DHV1,DHV,GHI,SSV} for several submanifolds in real, complex and quaternionic space forms.
In this paper, we establish two sharp inequalities involving the normalized scalar curvature and the generalized normalized $\delta$-Casorati curvatures for slant submanifolds in quaternionic space forms and also completely characterize those submanifolds for which the equality cases hold, generalizing some recent results from \cite{SSV}.

\section{Preliminaries}

This section gives several basic definitions and notations for our framework based mainly on \cite{CH01,ISH}.

Let $M^n$ be an $n$-dimensional Riemannian submanifold of an $m$-dimensional  Riemannian manifold $(\overline{M}^m,\overline{g})$. Then we denote by $g$ the metric
tensor induced on $M$. Let $K(\pi)$ be the sectional curvature of $M$ associated with a plane
section $\pi\subset T_pM,\ p\in M$. If $\{e_1,...,e_n\}$ is an
orthonormal basis of the tangent space $T_pM$ and $\{e_{n+1},...,e_{m}\}$ is an orthonormal basis of the normal space $T_p^\perp M$, then the scalar curvature
$\tau$ at $p$ is given by
\[
\tau(p)=\sum_{1\leq i<j\leq n}K(e_i\wedge e_j)
\]
and the normalized scalar curvature $\rho$ of $M$ is defined as
\[
\rho=\frac{2\tau}{n(n-1)}.
\]

If $\overline{\nabla}$ is the Levi-Civita connection on $\overline{M}$ and $\nabla$ is the covariant differentiation
induced on $M$, then the Gauss and Weingarten formulas are given by:
\[
       \overline{\nabla}_XY=\nabla_XY+h(X,Y), \forall X,Y \in
\Gamma(TM)
\]
and
\[
       \overline{\nabla}_XN=-A_NX+\nabla_{X}^{\perp}N, \forall X\in
\Gamma(TM), \forall N\in \Gamma(TM^\perp)
\]
where $h$ is the second fundamental form of $M$, $\nabla^\perp$ is
the connection on the normal bundle and $A_N$ is the shape operator
of $M$ with respect to $N$. If we denote by $\overline{R}$ and $R$ the curvature tensor fields
of $\overline{\nabla}$ and $\nabla$, then we have the Gauss equation:
\begin{eqnarray}\label{5}
\overline{R}(X,Y,Z,W)&=&R(X,Y,Z,W)+\overline{g}(h(X,W),h(Y,Z))\nonumber\\
&&-\overline{g}(h(X,Z),h(Y,W)),
       \end{eqnarray}
for all $X,Y,Z,W\in \Gamma(TM)$.

We denote by $H$ the mean curvature vector, that is
\[
H(p)=\frac{1}{n}\sum_{i=1}^{n}h(e_i,e_i)
\]
and we also set
\[
h_{ij}^\alpha=g(h(e_i,e_j),e_\alpha),\ i,j\in\{1,...,n\},\ \alpha\in\{n+1,...,m\}.
\]
Then it is well-known that the squared mean curvature of the submanifold $M$ in $\overline{M}$ is defined by
\[
\|H\|^2=\frac{1}{n^2}\sum_{\alpha=n+1}^{m}\left(\sum_{i=1}^{n}h_{ii}^\alpha\right)^2
\]
and the squared norm of $h$ over dimension $n$ is denoted by $\mathcal{C}$ and is called the Casorati curvature of the submanifold $M$. Therefore we have
\[
\mathcal{C}=\frac{1}{n}\sum_{\alpha=n+1}^{m}\sum_{i,j=1}^{n}\left(h_{ij}^\alpha\right)^2.
\]

The submanifold $M$ is called \emph{invariantly quasi-umbilical} if there exists $m-n$ mutually orthogonal unit normal vectors $\xi_{n+1},...,\xi_{m}$ such that the shape operators with respect to all directions $\xi_{\alpha}$ have an eigenvalue of multiplicity $n-1$ and that for each $\xi_{\alpha}$ the distinguished eigendirection is the same.

Suppose now that $L$ is an $s$-dimensional subspace of $T_pM$, $s\geq 2$ and
let $\{e_1,...,e_s\}$ be an orthonormal basis of $L$. Then the scalar
curvature $\tau(L)$ of the $s$-plane section $L$ is given by
\[
\tau(L)=\sum_{1\leq\alpha<\beta\leq s}K(e_\alpha\wedge e_\beta)
\]
and the Casorati curvature $\mathcal{C}(L)$ of the subspace $L$ is defined as
\[
\mathcal{C}(L)=\frac{1}{s}\sum_{\alpha=n+1}^{m}\sum_{i,j=1}^{s}\left(h_{ij}^\alpha\right)^2.
\]

The normalized  $\delta$-Casorati curvatures $\delta_c(n-1)$ and $\widehat{\delta}_c(n-1)$ of the submanifold $M^n$ are given by
\[
\left[\delta_c(n-1)\right]_p=\frac{1}{2}\mathcal{C}_p+\frac{n+1}{2n}{\rm inf}\{\mathcal{C}(L)|L\ {\rm a\ hyperplane\ of\ } T_pM\}
\]
and
\[
\left[\widehat{\delta}_c(n-1)\right]_p=2\mathcal{C}_p-\frac{2n-1}{2n}{\rm sup}\{\mathcal{C}(L)|L\ {\rm a\ hyperplane\ of\ } T_pM\}.
\]

The generalized normalized $\delta$-Casorati curvatures $\delta_C(r;n-1)$ and $\widehat{\delta}_C(r;n-1)$ of the submanifold $M^n$ are defined for any positive real number $r\neq n(n-1)$ as
\[
\left[\delta_C(r;n-1)\right]_p=r\mathcal{C}_p+\frac{(n-1)(n+r)(n^2-n-r)}{rn}{\rm inf}\{\mathcal{C}(L)|L\ {\rm a\ hyperplane\ of\ } T_pM\},
\]
if $0<r<n^2-n$, and
\[
\left[\widehat{\delta}_C(r;n-1)\right]_p=r\mathcal{C}_p-\frac{(n-1)(n+r)(r-n^2+n)}{rn}{\rm sup}\{\mathcal{C}(L)|L\ {\rm a\ hyperplane\ of\ } T_pM\},
\]
if $r>n^2-n$.

Assume now that $(\overline{M},\overline{g})$ is a smooth manifold such that
there is a rank 3-subbundle $\sigma$ of $End(T\overline{M})$ with local basis $\lbrace{J_1,J_2,J_3}\rbrace$ satisfying for all $\alpha\in\{1,2,3\}$:
    \[
    \overline{g}(J_{\alpha}\cdot,J_{\alpha}\cdot)=\overline{g}(\cdot,\cdot)
    \]
and
       \[
       J_\alpha^2=-{\rm Id},
       J_{\alpha}J_{\alpha+1}=-J_{\alpha+1}J_{\alpha}=J_{\alpha+2},
       \]
where ${\rm Id}$ denotes the identity tensor field of
type (1, 1) on $\overline{M}$ and the indices are taken from $\{1,2,3\}$ modulo 3.
Then $(\overline{M},\sigma,\overline{g})$ is said to be an almost
quaternionic Hermitian manifold. It is easy to see that such manifold is of dimension $4m$, $m\geq 1$.
Moreover, if the bundle $\sigma$ is parallel with respect to the Levi-Civita
connection $\overline{\nabla}$ of $\overline{g}$, then
$(\overline{M},\sigma,\overline{g})$ is said to be a quaternionic
K\"{a}hler manifold.

Let $(\overline{M},\sigma,\overline{g})$ be a quaternionic
K\"{a}hler manifold and let $X$ be a non-null vector field on
$\overline{M}$. Then the 4-plane spanned by $\{X,J_1X,J_2X,J_3X\}$,
denoted by $Q(X)$, is called a quaternionic 4-plane. Any 2-plane in
$Q(X)$ is called a quaternionic plane. The sectional curvature of a
quaternionic plane is called a quaternionic sectional curvature. A
quaternionic K\"{a}hler manifold is a quaternionic space form if its
quaternionic sectional curvatures are equal to a constant, say $c$.
It is well-known that a quaternionic K\"{a}hler manifold
$(\overline{M},\sigma,\overline{g})$ is a quaternionic space form,
denoted $\overline{M}(c)$, if and only if its curvature tensor is
given by
     \begin{eqnarray}\label{6}
       \overline{R}(X,Y)Z&=&\frac{c}{4}\lbrace\overline{g}(Z,Y)X-
       \overline{g}(X,Z)Y+\sum\limits_{\alpha=1}^3
        [\overline{g}(Z,J_\alpha Y)J_\alpha X-\nonumber\\
       &&-\overline{g}(Z,J_\alpha X)J_\alpha Y+
        2\overline{g}(X,J_\alpha Y)J_\alpha Z]\rbrace
      \end{eqnarray}
for all vector fields $X,Y,Z$ on $\overline{M}$ and any local basis
$\lbrace{J_1,J_2,J_3}\rbrace$ of $\sigma.$

A submanifold $M$ of a quaternionic K\"{a}hler manifold
$(\overline{M},\sigma,\overline{g})$ is said to be a slant submanifold \cite{SH} if for each
non-zero vector $X$ tangent to $M$ at $p$, the angle $\theta(X)$
between $J_\alpha(X)$ and $T_pM$, $\alpha\in\{1,2,3\}$ is constant,
i.e. it does not depend on the choice of $p\in M$ and $X\in T_pM$.
We can easily see that quaternionic submanifolds are slant
submanifolds with $\theta=0$ and totally-real submanifolds are slant
submanifolds with $\theta=\frac{\pi}{2}$. A slant submanifold of a
quaternionic K\"{a}ler manifold is said to be proper (or
$\theta$-slant proper) if it is neither quaternionic nor totally
real. We recall that every proper slant submanifold of a quaternionic K\"{a}hler manifold
is of even dimension $n=2s\geq 2$ and we can
choose a canonical orthonormal local frame, called an adapted slant frame, as follows:
\[
\{e_1,e_2={\rm sec}\theta\ P_\alpha
e_1,...,e_{2s-1},e_{2s}={\rm sec}\theta\ P_\alpha e_{2s-1}\},
\]
where $P_{\alpha}e_{2k-1}$ denotes the tangential component of
$J_{\alpha}e_{2k-1}$, $k\in\{1...,s\}$ and $\alpha$ is 1, 2 or 3 (see \cite{VIL1,VIL3}).

\section{Main results}

\begin{thm}\label{T1}
Let $M^n$ be a $\theta$-slant proper submanifold of a quaternionic
space form $\overline{M}^{4m}(c)$. Then:
\begin{enumerate}
  \item[(i)] The generalized normalized $\delta$-Casorati curvature $\delta_C(r;n-1)$ satisfies
  \begin{equation}\label{1}
\rho\leq\frac{\delta_C(r;n-1)}{n(n-1)}+\frac{c}{4}\left(1+\frac{9}{n-1}\cos^2\theta\right)
\end{equation}
for any real number $r$ such that $0<r<n(n-1)$.
  \item[(ii)] The generalized normalized $\delta$-Casorati curvature $\widehat{\delta}_C(r;n-1)$ satisfies
  \begin{equation}\label{2}
\rho\leq\frac{\widehat{\delta}_C(r;n-1)}{n(n-1)}+\frac{c}{4}\left(1+\frac{9}{n-1}\cos^2\theta\right)
\end{equation}
for any real number $r>n(n-1)$.
\end{enumerate}

Moreover, the equality sign holds in the inequalities (\ref{1}) and (\ref{2}) if and only if $M^n$ is an invariantly quasi-umbilical submanifold with trivial normal connection in  $\overline{M}^{4m}(c)$, such that with respect to suitable orthonormal tangent frame $\{\xi_1,...,\xi_{n}\}$  and normal orthonormal frame  $\{\xi_{n+1},...,\xi_{4m}\}$, the shape operators $A_r\equiv A_{\xi_r}$, $r\in\{n+1,...,4m\}$,
take the following forms:
\begin{equation}\label{3}
A_{n+1}= \left( \begin{array}{cccccc}
a & 0 & 0 & ... & 0 & 0\\
0 & a & 0 & ... & 0 & 0\\
0 & 0 & a & ... & 0 & 0\\
\vdots & \vdots & \vdots & \ddots & \vdots & \vdots\\
0 & 0 & 0 & ... & a & 0\\
0 & 0 & 0 & ... & 0 & \frac{n(n-1)}{r}a \end{array} \right),\ A_{n+2}=...=A_{4m}=0.
\end{equation}
\end{thm}
\begin{proof}
Since $M$ is $\theta$-slant, then it is known from \cite{SH} that
\begin{equation}\label{7}
P_{\beta}P_{\alpha}X=-{\rm cos}^2\theta X,\ \ \forall X\in\Gamma(TM),\
\alpha,\beta\in\{1,2,3\},
\end{equation}
where $P_{\alpha}X$ denotes the tangential component of
$J_{\alpha}X$.

From (\ref{7}) it follows immediately that
\begin{equation}\label{8}
g(P_\alpha X,P_\beta Y)={\rm cos}^2\theta g(X,Y).
\end{equation}
for $X,Y\in \Gamma(TM)$ and
$\alpha,\beta\in\{1,2,3\}$.

On the other hand, because $\overline{M}^{4m}(c)$ is a quaternionic space form, from
(\ref{5}) and (\ref{6}) we derive
\begin{equation}\label{9}
n^2\|H\|^2=2\tau(p)+\|h\|^2-\frac{n(n-1)c}{4}-\frac{3c}{4}\sum_{\beta=1}^3\sum_{i,j=1}^{n}g^2(P_\beta
e_i,e_j).
\end{equation}

Choosing now an
adapted slant basis$\{e_1,e_2={\rm sec}\theta\ P_\alpha
e_1,...,e_{2s-1},e_{2s}={\rm sec}\theta\ P_\alpha e_{2s-1}\}$
of $T_pM$, $p\in M$,
where $2s=n$, from (\ref{7}) and (\ref{8}), we derive
\begin{equation}\label{10}
g^2(P_\beta e_i,e_{i+1})=g^2(P_\beta e_{i+1},e_i)=\cos^2\theta,\ {\rm
for}\ i=1,3,...,2s-1
\end{equation}
and
\begin{equation}\label{11}
g(P_\beta e_i,e_j)=0,\ {\rm for}\
(i,j)\not\in\{(2l-1,2l),(2l,2l-1)|l\in\{1,2,....,s\}\}.
\end{equation}

By using (\ref{10}) and (\ref{11}) in (\ref{9})  we get
\begin{equation}\label{12}
2\tau(p)=n^2\|H\|^2-n\mathcal{C}+\frac{c}{4}[n(n-1)+9n\cos^2\theta].
\end{equation}

We consider now the following quadratic polynomial in the components of the second fundamental form:
\[
\mathcal{P}=r\mathcal{C}+\frac{(n-1)(n+r)(n^2-n-r)}{rn}\mathcal{C}(L)-2\tau(p)+\frac{c}{4}\left[n(n-1)+9n\cos^2\theta\right],
\]
where $L$ is a hyperplane of $T_pM$.  Without loss of generality we can assume that $L$ is spanned by $e_1,...,e_{n-1}$. Then we derive
\begin{eqnarray}\label{13}
\mathcal{P}&=&\frac{r}{n}\sum_{\alpha=n+1}^{4m}\sum_{i,j=1}^{n}\left(h_{ij}^\alpha\right)^2
+\frac{(n+r)(n^2-n-r)}{rn}\sum_{\alpha=n+1}^{4m}\sum_{i,j=1}^{n-1}\left(h_{ij}^\alpha\right)^2\nonumber\\
&&-2\tau(p)+\frac{c}{4}\left[n(n-1)+9n\cos^2\theta\right].
\end{eqnarray}

From (\ref{12}) and (\ref{13}), we obtain
\begin{eqnarray}
\mathcal{P}&=&\frac{n+r}{n}\sum_{\alpha=n+1}^{4m}\sum_{i,j=1}^{n}\left(h_{ij}^\alpha\right)^2\nonumber\\
&&+\frac{(n+r)(n^2-n-r)}{rn}\sum_{\alpha=n+1}^{4m}\sum_{i,j=1}^{n-1}\left(h_{ij}^\alpha\right)^2
-\sum_{\alpha=n+1}^{4m}\left(\sum_{i=1}^{n}h_{ii}^\alpha\right)^2\nonumber
\end{eqnarray}
which is equivalent to
\begin{eqnarray}\label{14}
\mathcal{P}&=&\sum_{\alpha=n+1}^{4m}\sum_{i=1}^{n-1}\left[\frac{n^2+n(r-1)-2r}{r}\left(h_{ii}^\alpha\right)^2+\frac{2(n+r)}{n}\left(h_{in}^\alpha\right)^2\right]\nonumber\\
&&+\sum_{\alpha=n+1}^{4m}\left[\frac{2(n+r)(n-1)}{r}\sum_{i<j=1}^{n-1}\left(h_{ij}^\alpha\right)^2-2\sum_{i<j=1}^{n}h_{ii}^\alpha h_{jj}^\alpha+\frac{r}{n}\left(h_{nn}^{\alpha}\right)^2\right].
\end{eqnarray}

From (\ref{14}) it follows that critical points
\[h^c=\left(h_{11}^{n+1},h_{12}^{n+1},...,h_{nn}^{n+1},...,h_{11}^{4m},h_{12}^{4m},...,h_{nn}^{4m}\right)\]
of $\mathcal{P}$ are the solutions of the following system of linear homogeneous equations:
\begin{equation}\label{15}
\begin{cases} \frac{\partial\mathcal{P}}{\partial h_{ii}^\alpha}= \frac{2(n+r)(n-1)}{r}h_{ii}^\alpha-2\displaystyle\sum_{k=1}^{n}h_{kk}^\alpha=0\\
\frac{\partial\mathcal{P}}{\partial h_{nn}^\alpha}=\frac{2r}{n}h_{nn}^\alpha-2\displaystyle\sum_{k=1}^{n-1}h_{kk}^\alpha=0 \\
\frac{\partial\mathcal{P}}{\partial h_{ij}^\alpha}=\frac{4(n+r)(n-1)}{r}h_{ij}^\alpha=0 \\
\frac{\partial\mathcal{P}}{\partial h_{in}^\alpha}=\frac{4(n+r)}{n}h_{in}^\alpha=0 \end{cases}
\end{equation}
with $i,j\in\{1,...,n-1\}$, $i\neq j$, and  $\alpha\in\{n+1,...,4m\}$.

From (\ref{15}) it follows that every solutions $h^c$ has $h_{ij}^{\alpha}=0$ for $i\neq j$ and the determinant which corresponds to the first two sets of equations of the above system is zero (there exist solutions for non-totally geodesic submanifolds). Moreover it is easy to see that the Hessian matrix of $\mathcal{P}$ has the form
\begin{equation}
\mathcal{H}(\mathcal{P})= \left( \begin{array}{ccc}
 H_1 & \mathbf{0} & \mathbf{0}\\
 \mathbf{0} & H_2 & \mathbf{0}\\
 \mathbf{0} & \mathbf{0} & H_3 \end{array} \right),\nonumber
\end{equation}
where
\begin{equation}
H_1= \left( \begin{array}{ccccc}
\frac{2(n+r)(n-1)}{r}-2 & -2  & ... & -2 & -2 \\
-2 & \frac{2(n+r)(n-1)}{r}-2 &  ... & -2 & -2\\
\vdots & \vdots  & \ddots & \vdots & \vdots\\
-2 & -2 &  ... & \frac{2(n+r)(n-1)}{r}-2 & -2\\
-2 & -2  & ... & -2 &  \frac{2r}{n}\\\end{array} \right),\nonumber
\end{equation}
$\mathbf{0}$ denotes the null matrix of corresponding dimensions and $H_2,H_3$ are the next diagonal matrices
\[
H_2={\rm diag}\left(\frac{4(n+r)(n-1)}{r},\frac{4(n+r)(n-1)}{r},\ldots,\frac{4(n+r)(n-1)}{r}\right),
\]
\[
H_3={\rm diag}\left(\frac{4(n+r)}{n},\frac{4(n+r)}{n},\ldots,\frac{4(n+r)}{n}\right).
\]

Therefore we find that $\mathcal{H}(\mathcal{P})$ has the following eigenvalues:
\[
\lambda_{11}=0,\ \lambda_{22}=\frac{2(n^3-n^2+r^2)}{rn},\ \lambda_{33}=...=\lambda_{nn}=\frac{2(n+r)(n-1)}{r},
\]
\[
\lambda_{ij}=\frac{4(n+r)(n-1)}{r},\ \lambda_{in}=\frac{4(n+r)}{n}, \forall i,j\in\{1,...,n-1\},\ i\neq j.
\]

Hence we deduce that $\mathcal{P}$ is parabolic and reaches a minimum $\mathcal{P}(h^c)$ for each solution $h^c$ of the system (\ref{15}). Inserting now (\ref{15}) in (\ref{14}) we get that  $\mathcal{P}(h^c)=0$. So $\mathcal{P}\geq 0$, and this implies
\[
2\tau(p)\leq r\mathcal{C}+\frac{(n-1)(n+r)(n^2-n-r)}{rn}\mathcal{C}(L)+\frac{c}{4}\left[n(n-1)+9n\cos^2\theta\right]
\]

Therefore we derive
\begin{equation}\label{AAA}
\rho\leq\frac{r}{n(n-1)}\mathcal{C}+\frac{(n+r)(n^2-n-r)}{rn^2}\mathcal{C}(L)+\frac{c}{4}\left(1+\frac{9}{n-1}\cos^2\theta\right)
\end{equation}
for every tangent hyperplane $L$ of $M$ and both inequalities (\ref{1}) and (\ref{2}) obviously follow from (\ref{AAA}).

Moreover, we can easily see now that the equality sign holds in the inequalities (\ref{1}) and (\ref{2}) if and only if
\begin{equation}\label{16}
h_{ij}^\alpha=0,\ \forall\ i,j\in\{1,...,n\},\ i\neq j\,
\end{equation}
and
\begin{equation}\label{17}
h_{nn}^\alpha=\frac{n(n-1)}{r}h_{11}^\alpha=\frac{n(n-1)}{r}h_{22}^\alpha=...=\frac{n(n-1)}{r}h_{n-1,n-1}^\alpha
\end{equation}
for all $\alpha\in\{n+1,...,4m\}$.

Finally, from (\ref{16}) and (\ref{17}) we deduce that the equality sign holds in (\ref{1}) and (\ref{2}) if and only if the submanifold $M$ is invariantly quasi-umbilical with trivial normal connection in  $\overline{M}$, such that the shape operators take the forms (\ref{3}) with respect to suitable tangent and normal orthonormal frames.
\end{proof}

\begin{cor}
Let $M^n$ be a $\theta$-slant proper submanifold of a quaternionic
space form $\overline{M}^{4m}(c)$. Then:
\begin{enumerate}
  \item[(i)] The normalized $\delta$-Casorati curvature $\delta_c(n-1)$ satisfies
  \begin{equation}\label{b1}
\rho\leq\delta_c(n-1)+\frac{c}{4}\left(1+\frac{9}{n-1}\cos^2\theta\right).
\end{equation}

Moreover, the equality sign holds if and only if $M^n$ is an invariantly quasi-umbilical submanifold with trivial normal connection in  $\overline{M}^{4m}(c)$, such that with respect to suitable orthonormal tangent frame $\{\xi_1,...,\xi_{n}\}$  and normal orthonormal frame  $\{\xi_{n+1},...,\xi_{4m}\}$, the shape operators $A_r\equiv A_{\xi_r}$, $r\in\{n+1,...,4m\}$,
take the following forms:
\begin{equation}\label{b2}
A_{n+1}= \left( \begin{array}{cccccc}
a & 0 & 0 & ... & 0 & 0\\
0 & a & 0 & ... & 0 & 0\\
0 & 0 & a & ... & 0 & 0\\
\vdots & \vdots & \vdots & \ddots & \vdots\\
0 & 0 & 0 & ... & a & 0\\
0 & 0 & 0 & ... & 0 & 2a \end{array} \right),\ A_{n+2}=...=A_{4m}=0.
\end{equation}

  \item[(ii)] The normalized $\delta$-Casorati curvature $\widehat{\delta}_c(n-1)$ satisfies
  \begin{equation}\label{b3}
\rho\leq\widehat{\delta}_c(n-1)+\frac{c}{4}\left(1+\frac{9}{n-1}\cos^2\theta\right).
\end{equation}

Moreover, the equality sign holds if and only if $M^n$ is an invariantly quasi-umbilical submanifold with trivial normal connection in  $\overline{M}^{4m}(c)$, such that with respect to suitable orthonormal tangent frame $\{\xi_1,...,\xi_{n}\}$  and normal orthonormal frame  $\{\xi_{n+1},...,\xi_{4m}\}$, the shape operators $A_r\equiv A_{\xi_r}$, $r\in\{n+1,...,4m\}$,
take the following forms:
\begin{equation}\label{b4}
A_{n+1}= \left( \begin{array}{cccccc}
2a & 0 & 0 & ... & 0 & 0\\
0 & 2a & 0 & ... & 0 & 0\\
0 & 0 & 2a & ... & 0 & 0\\
\vdots & \vdots & \vdots & \ddots & \vdots\\
0 & 0 & 0 & ... & 2a & 0\\
0 & 0 & 0 & ... & 0 & a \end{array} \right),\ A_{n+2}=...=A_{4m}=0.
\end{equation}
\end{enumerate}
\end{cor}
\begin{proof}
(i) It is easy to see that the following relation holds
\begin{equation}\label{50}
\left[\delta_C\left(\frac{n(n-1)}{2};n-1\right)\right]_p=n(n-1)\left[\delta_C(n-1)\right]_p
\end{equation}
in any point $p\in M$. Therefore, taking $r=\frac{n(n-1)}{2}$ in (\ref{1}) and making use of (\ref{50}) we obtain the conclusion.\\
(ii) The following relation can be easily verified:
\begin{equation}\label{51}
\left[\widehat{\delta}_C\left(2n(n-1);n-1\right)\right]_p=n(n-1)\left[\widehat{\delta}_c(n-1)\right]_p,\ \forall p\in M.
\end{equation}

Replacing now $r=2n(n-1)$ in (\ref{2}) and taking account of (\ref{51}) we derive the conclusion.
\end{proof}

\begin{rem}
We note that the above Corollary was recently proved in \cite{SSV}, but with
a slightly modified coefficient in the definition of $\delta_C(n-1)$; in fact, in \cite{SSV} it was used the coefficient $\frac{n+1}{2n(n-1)}$, as in \cite{DHV1,GHI}, instead of $\frac{n+1}{2n}$, like in the present paper. However, because the normalized $\delta$-Casorati curvature $\delta_C(n-1)$ should be able to be recovered from the generalized  normalized $\delta$-Casorati curvature $\delta_C(r;n-1)$ for a positive real number $r\neq n(n-1)$, it would be more appropriate to define $\delta_C(n-1)$ using the coefficient $\frac{n+1}{2n}$ and therefore the proof of Theorem 1.1 (i) in \cite{SSV}  should be adapted to the amended coefficient.  But we would like to point out that this can be done easily replacing by
$\frac{(n+1)(n-1)}{2}$ the coefficient $\frac{n+1}{2}$ of $\mathcal{C}(L)$ in the
definition of a quadratic polynomial $\mathcal{P}$ between (12) and (13) from \cite{SSV} and modifying the corresponding coefficients in (13), (14) and (15). Hence the eigenvalues of the new Hessian matrix of $\mathcal{P}$  become
\[
\lambda_{11}=0,\ \lambda_{22}=n+3,\ \lambda_{33}=...=\lambda_{nn}=2(n+1),
\]
\[
\lambda_{ij}=4(n+1),\ \lambda_{in}=2(n+1), \forall i,j\in\{1,...,n-1\},\ i\neq j.
\]
and the rest of the proof in \cite{SSV} remains unchanged.
\end{rem}

\section*{Acknowledgement} The second author was supported by
CNCS-UEFISCDI, project number PN-II-ID-PCE-2011-3-0118.

Jaewon LEE\\
Yeungnam University,\\
School of General Education,\\
Gyeongsan, 712-749, South Korea\\
E-mail: \emph{leejaew@yu.ac.kr}\\

Gabriel Eduard V\^{I}LCU$^{1,2}$ \\
      $^1$Petroleum-Gas University of Ploie\c sti\\
         Department of Mathematical Modelling, Economic Analysis and Statistics\\
         Bulevardul Bucure\c sti, Nr. 39, Ploie\c sti 100680, Romania\\
         e-mail: \emph{gvilcu@upg-ploiesti.ro}\\
      $^2$University of Bucharest\\
          Faculty of Mathematics and Computer Science\\
         Research Center in Geometry, Topology and Algebra\\
         Str. Academiei, Nr. 14, Sector 1, Bucure\c sti 70109, Romania\\
         e-mail: \emph{gvilcu@gta.math.unibuc.ro}

\end{document}